\documentclass[%
 aip,
 amsmath,amssymb,
 reprint,%
jmp]{revtex4-2}

\usepackage{graphicx}% Include figure files
\usepackage{dcolumn}% Align table columns on decimal point
\usepackage{bm}% bold math
\usepackage[mathlines]{lineno}% Enable numbering of text and display math
\usepackage{amsthm}
\usepackage{amsmath}
\usepackage{amssymb}
\usepackage{mathrsfs}%
\usepackage{mathtools}
\usepackage{mathptmx}
\usepackage{etoolbox}

\makeatletter
\def\@email#1#2{%
 \endgroup
 \patchcmd{\titleblock@produce}
  {\frontmatter@RRAPformat}
  {\frontmatter@RRAPformat{\produce@RRAP{*#1\href{mailto:#2}{#2}}}\frontmatter@RRAPformat}
  {}{}
}%
\makeatother

\newtheorem{theorem}{Theorem}[section]
\newtheorem*{definition}{Definition}
\newtheorem{corollary}{Corollary}[theorem]
\newtheorem{proposition}{Proposition}[theorem]
\newtheorem{lemma}{Lemma}[theorem]

\begin{document}

\preprint{AIP/123-QED}

\title[Smooth Functional Calculus and Spectral Theorem in Banach Spaces]{Smooth Functional Calculus and Spectral Theorem in Banach Spaces}
\author{{Luis A.} {Cede\~no-P\'erez} and Hernando Quevedo}

\email{luisacp@ciencias.unam.mx,quevedo@nucleares.unam.mx}

\affiliation{ Instituto de Ciencias Nucleares, Universidad Nacional Aut\'onoma de M\'exico, AP  70543, Mexico City, Mexico }
\affiliation{ Dipartimento di Fisica and Icra, Universit\`a di Roma “La Sapienza”, Roma, Italy }
\affiliation{ Al-Farabi Kazakh National University, Al-Farabi av. 71, 050040 Almaty, Kazakhstan}

\date{\today}

\begin{abstract}
The notion of projection families generalizes the classical notions of vector- and operator-valued measures. We show that projection families are general enough to extend the Spectral Theorem to Banach algebras and operators between Banach spaces. To this end, we first develop a Smooth Functional Calculus in Banach algebras using the Cauchy-Pompeiu formula, which is further extended to a Continuous Functional Calculus. We also show that these theorems are proper generalizations of the usual result for operators between Hilbert spaces.
\end{abstract}


\maketitle

\tableofcontents

\section{Introduction}

In a previous work \cite{ProjectionFamilies}, we defined a new kind of measures called \textbf{projection families} and developed their theory of integration. In the present article, we show that the properties of projection families permit us to generalize the Spectral Theorem, both to Banach algebras and operators between Banach spaces.

Let $X$ be a Banach algebra, and $x\in X$. If $f$ is an holomorphic function in a neighborhood $\Omega$ of $\sigma(x)$, we define $f(x)$ as the element of $X$ given by
\begin{equation*}
    f(x) = \frac{1}{2\pi i}\int_{\gamma}f(\lambda)\;R_{x}(\lambda)\;d\lambda,
\end{equation*}
where $\gamma$ is a simple closed curve $\Omega$ that surrounds $\sigma(x)$ and $R_{x}(\lambda)$ is the resolvent function of $x$. This is known as the \textbf{Holomorphic Functional Calculus}. Despite being defined for every element of the Banach algebra, this functional calculus has the defect that the class of holomorphic functions is too small.

The Holomorphic Functional Calculus is clearly based on the Cauchy Integral Formula. If we wish to extend the Holomorphic Functional Calculus to the continuous functions in $\sigma(x)$, it is natural to first look for a generalization of the Cauchy Integral Formula to use as a starting point. We use the less known \textbf{Cauchy-Pompeiu Formula}, which generalizes the Cauchy Integral Formula to smooth functions. This formula states that if $\Omega\subset\mathbb{C}$ is an open set with compact closure and smooth boundary and $f$ is smooth in $\overline{\Omega}$ then
\begin{equation*}
    f(\lambda) = \frac{1}{2\pi i}\int_{\partial\omega}\frac{f(z)}{z-\lambda}\;dz - \frac{1}{\pi}\int_{\omega}\frac{\partial_{z^{\ast}}f(z)}{z-\lambda}\;dx dy
\end{equation*}
for any $\lambda\in\Omega$ (see \cite{Hormander}). Based on the Cauchy-Pompeiu Formula, the Smooth Functional Calculus should be defined as
\begin{equation*}
    f(x) = \frac{1}{2\pi i}\int_{\gamma}f(\lambda)R_{x}(\lambda)\;d\gamma(\lambda) - \frac{1}{\pi}\int_{int\;\gamma\setminus \sigma(x)}\partial_{z^{\ast}}f(\lambda)R_{x}(\lambda)\;d\ell(\lambda).
\end{equation*}
The first integral exists by continuity and compactness, however, the second one is much more complicated since the resolvent function $R_{x}(\lambda)$ diverges close to the spectrum. The first integral ignores this as the curve $\gamma$ is never in the spectrum, but this cannot be avoided in the second integral since the integration is in the plane. The second integral is, essentially, a singular integral and as such may not define a continuous operator. Two things are needed to deal with this singular integral. Firstly, we require $x$ to have additional properties in terms of its resolvent $R_{x}$. This leads to the notion of \textbf{regular elements}, which are precisely those singular integrals with a continuous behavior. The second is to allow $f(x)$ to be an element of the larger space $X^{\ast\ast}$ instead of $X$. The previous formula gives the definition of the \textbf{Smooth Functional Calculus} when both conditions are met.

Given that smooth functions are dense in the continuous functions, one would expect the extension to the continuous case to be rather straightforward. This is not the case since the Smooth Functional Calculus is easily seen to be continuous with respect to the norm of $C^{1}$, but the continuity required to extend by density is with respect to the uniform norm. This difficulty comes from, essentially, the appearance of the first derivatives in the Smooth Functional Calculus. To deal with this it is necessary to manipulate the definition of the Smooth Functional Calculus in such a way that the resulting expression does not have a dependence on first derivatives of the function. Once this is done, the continuous extension to the space of continuous functions will be possible, resulting in the \textbf{Continuous Functional Calculus}.

The existence of the Continuous Functional Calculus will allow us to prove two versions of the Spectral Theorem, the first for Banach algebras and the second for operators between Banach spaces. The Banach algebra version is clearly valid in the Banach algebra $B(X)$, however, it turns out to be too restrictive, thus the second version is developed under less restrictive hypotheses which we show are enough for operators. The central idea of our constructions is to study the function
\begin{equation*}
    (f,\Lambda,x) \longmapsto \Lambda(f(T)(x)),
\end{equation*}
for a fixed operator $T\in B(X)$, and study its continuity properties. For fixed $\Lambda$ and $x$, the continuity with respect to $f$ leads to the existence of a measure $\mu^{T}_{\Lambda,x}$ which determines the functional. The continuity with respect to $\Lambda$ and $x$ implies that the family of measures
\begin{equation*}
    \mu^{T} = \{\mu^{T}_{\Lambda,x}\;|\;\Lambda\in X^{\ast},\;x\in X\}
\end{equation*}
is an operator projection family. This allows us to extend the Continuous Functional Calculus to the space of integrable functions $L^{1}(\mu^{T})$. Finally, we show that this theorem is a strict generalization of the usual Spectral Theorem for operators between Hilbert spaces.

The first section of this work is to prove the Cauchy-Pompeiu Formula and other results in geometric integration that will be useful throughout the work. In the second section, we state the basic results on vector integration and projection families, as well as the analytical tools we will require. In the third section, we develop the Smooth and Continuous Functional Calculi executing the program described in this section. Finally, in the last section, we prove the Spectral Theorem in the contexts of Banach algebras and operators between Banach spaces.

\section{Geometric Integration}

\subsection{Cauchy-Pompeiu Formula}

The complex plane $\mathbb{C}$ has two coordinate differential forms $dx$ and $dy$, which induce linearly independent differential forms
\begin{equation*}
    dz = dx + i\;dy
\end{equation*}
and
\begin{equation*}
    dz^{\ast} = dx - i\;dy,
\end{equation*}
and thus span every $1$-form in $\mathbb{C}$. A simple computation shows that
\begin{equation}\label{EqLebesgueNoLebesgue}
    dz \wedge dz^{\ast} = -2i\;dx\wedge dy,
\end{equation}
and every $2$-form is obtained by multiplying this form by a scalar function. These differential forms have associated tangent vectors
\begin{equation*}
    2\partial_{z} = \partial_{x} - i\partial_{y}
\end{equation*}
and
\begin{equation*}
    2\partial_{z^{\ast}} = \partial_{x} + i\partial_{y},
\end{equation*}
given by the conditions $dz(\partial_{z}) = 1$, $dz^{\ast}(\partial_{z^{\ast}}) = 1$ and $dz(\partial_{z^{\ast}}) = dz^{\ast}(\partial_{z}) = 0$.

If $u$ is a smooth function defined in an open set $\omega$ with smooth boundary $\partial\omega$, then the complex integral along $\partial\omega$ coincides with the integral of the $1$-form $u\;dz$ along $\partial\omega$. Furthermore, the Stokes Theorem implies that
\begin{align}\label{EqStokesCompleja}
    \notag \int_{\partial\omega}u\;dz &= \int_{\omega}d(u\;dz)\\
    \notag &= \int_{\omega}du\wedge dz\\
    \notag &= \int_{\omega}(\partial_{z}u\;dz + \partial_{z^{\ast}}u\;dz^{\ast})\wedge dz\\
    &= \int_{\omega}\partial_{z^{\ast}}u\;dz^{\ast}\wedge dz.
\end{align}
It is convenient to note that the function
\begin{equation*}
    \lambda \longmapsto \frac{1}{\lambda}
\end{equation*}
is integrable near the origin, since the dimension is two. This simplifies the convergence of certain integrals.

\begin{theorem}[Cauchy-Pompeiu Formula]\label{CauchyPompeiu}
Let $\Omega\subset\mathbb{C}$ be an open set with compact closure, $u\in C^{\infty}(\Omega)$ and $\omega\subset\Omega$ an open set with smooth boundary such that $\partial\omega\subset\Omega$. The formula
\begin{align*}
    u(\lambda) &= \frac{1}{2\pi i}\int_{\partial\omega}\frac{u}{z-\lambda}\;dz + \frac{1}{2\pi i}\int_{\omega}\frac{\partial_{z^{\ast}}u}{z-\lambda} dz^{\ast}\wedge dz\\
    &= \frac{1}{2\pi i}\int_{\partial\omega}\frac{u}{z-\lambda}\;dz - \frac{1}{\pi}\int_{\omega}\frac{\partial_{z^{\ast}}u}{z-\lambda}\;dx\wedge dy
\end{align*}
is valid for any $\lambda\in\omega$.
\end{theorem}
\begin{proof}
Since $\lambda\in\omega$ and $\omega$ is open, there exists $\epsilon > 0$ such that $B_{\epsilon}(\lambda)\subset\omega$. We define
\begin{equation*}
    \omega_{\epsilon} = \omega\setminus B_{\epsilon}(\lambda).
\end{equation*}
The function
\begin{equation*}
    \frac{u}{z-\lambda}
\end{equation*}
is $C^{\infty}$ in $\omega_{\epsilon}$, thus equation (\ref{EqStokesCompleja}) implies that
\begin{align}\label{EqCauchyPompeiu1}
    \notag\int_{\omega_{\epsilon}}\frac{\partial_{z^{\ast}}u}{z - \lambda}\;dz^{\ast}\wedge dz &= \int_{\partial\omega_{\epsilon}}\frac{u}{z-\lambda}\;dz\\
    &= \int_{\partial\omega}\frac{u}{z-\lambda}\;dz - \int_{\partial B_{\epsilon}(\lambda)}\frac{u}{z-\lambda}\;dz.
\end{align}
Parametrize $\partial B_{\epsilon}(\lambda)$ as $\lambda + \epsilon e^{it}$ with $t\in [0,2\pi]$, hence
\begin{equation*}
    dz = i\epsilon e^{it}\;dt.
\end{equation*}
It follows that
\begin{align*}
    \int_{\partial B_{\epsilon}(\lambda)}\frac{u}{z-\lambda}\;dz &= i\int_{0}^{2\pi}u(\lambda + \epsilon e^{it})\;dt.
\end{align*}
Substituting in equation (\ref{EqCauchyPompeiu1}) we find that
\begin{equation*}
    \int_{\omega_{\epsilon}}\frac{\partial_{z^{\ast}}u}{z - \lambda}\;dz^{\ast}\wedge dz = \int_{\partial\omega}\frac{u}{z-\lambda}\;dz - i\int_{0}^{2\pi}u(\lambda + \epsilon e^{it})\;dt.
\end{equation*}
The Monotone Convergence Theorem on the left-hand side and the Dominated Convergence Theorem on the right-hand side imply that the limit as $\epsilon\to 0$ is
\begin{equation*}
    \int_{\omega}\frac{\partial_{z^{\ast}}u}{z - \lambda}\;dz^{\ast}\wedge dz = \int_{\partial\omega}\frac{u}{z-\lambda}\;dz - 2\pi i\;u(\lambda).
\end{equation*}
The result follows from rearranging terms.
\end{proof}

The language of differential topology is useful to deduce the Cauchy-Pompeiu Formula, however, in the rest of the work the language of measure theory will be more useful.

\begin{corollary}
Under the hypotheses of the Cauchy-Pompeiu Formula, if $\gamma = \partial\omega$, $d\gamma$ is the measure induced by $\gamma$ and $\ell$ is the Lebesgue measure then
\begin{equation*}
    u(\lambda) = \frac{1}{2\pi i}\int_{\partial\omega}\frac{u}{z-\lambda}\;d\gamma - \frac{1}{\pi}\int_{\omega}\frac{\partial_{z^{\ast}}u}{z-\lambda}\;d\ell.
\end{equation*}
\end{corollary}
\begin{proof}
It follows from the Cauchy-Pompeiu Formula and equation (\ref{EqLebesgueNoLebesgue}).
\end{proof}

Just as for the Cauchy Integral Formula and analytic functions, the Cauchy-Pompeiu Formula is equivalent to other statements on the nullity of integrals and the independence of trajectories.

\begin{theorem}[Cauchy-Pompeiu]
Let $\Omega$ be an open subset of $\mathbb{C}$ with compact closure and $f$ a smooth function in $\Omega$. The following are satisfied:
\begin{enumerate}
    \item For any simple closed curve $\gamma$ in $\Omega$ we have that
    \begin{equation*}
        \frac{1}{2\pi i}\int_{\gamma} f\;d\gamma - \frac{1}{\pi}\int_{int\;\gamma}\partial_{z^{\ast}}f\;d\ell = 0,
    \end{equation*}
    where $int\;\gamma$ is the geometric interior of $\gamma$ and not the topological interior.
    \item If $\gamma_{1}$ and $\gamma_{2}$ are closed curves homotopic in $\Omega$ then
    \begin{equation*}
        \frac{1}{2\pi i}\int_{\gamma_{1}} f\;d\gamma - \frac{1}{\pi}\int_{int\;\gamma_{1}}\partial_{z^{\ast}}f\;d\ell = \frac{1}{2\pi i}\int_{\gamma_{2}} f\;d\gamma - \frac{1}{\pi}\int_{int\;\gamma_{2}}\partial_{z^{\ast}}f\;d\ell.
    \end{equation*}
\end{enumerate}
\end{theorem}
\begin{proof}
For the first statement, we apply the Cauchy-Pompeiu Formula to the smooth function $f(z)(z-\lambda)$ with $\lambda\in int\;\gamma$. For the second statement, we apply the first one to an appropriate homology class.
\end{proof}

\subsection{Coarea Formula}

The last fact from geometric integration that we require is the Coarea Formula, which relates the integrals on sublevel sets and integrals on level sets (see \cite{Chavel}). This formula will allow us to study the singular behavior of the resolvent function $R_{x}$, which in turn is related to the singular behavior of reciprocals of distance-to-a-set functions.

\begin{theorem}[Coarea Formula]\label{FórmulaCoárea}
Let $\Omega\subset\mathbb{C}$ be an open set and $H\colon \Omega \to [0,1]$ a $C^{\infty}$ function. For each $t\in [0,1]$ define
\begin{equation*}
    \Omega_{t} = H^{-1}([0,t))
\end{equation*}
and
\begin{equation*}
    \gamma_{t} = H^{-1}(\{t\}).
\end{equation*}
If $\Phi_{t}$ is the flow of $\nabla H$ then
\begin{equation*}
    d\ell = \frac{1}{|\nabla H|} d\gamma_{t}\wedge d_{t}\Phi,
\end{equation*}
thus, if $f\colon \Omega\to\mathbb{R}$ is integrable then
\begin{equation*}
    \int_{\Omega}f\;d\ell = \int_{0}^{1}\int_{\gamma_{t}}f\frac{1}{|\nabla H|}\;d\gamma_{t}\;dt.
\end{equation*}
\end{theorem}
\begin{proof}
We first note that by Sard's Theorem, the set of critical values of $H$ has zero Lebesgue measure; thus, the curves $\gamma_{t}$ are smooth except for a null set. The Chain Rule implies that the differential form $d_{t}\Phi$ satisfies the equations
\begin{equation*}
    d_{t}\Phi(\dot{\gamma_{t}}) = 0
\end{equation*}
and
\begin{equation*}
    |d_{t}\Phi| = |\nabla H|,
\end{equation*}
hence if $\gamma_{t}$ is parametrized by arc-length and $\vec{v}$ is a vector parallel to $\nabla H$ then
\begin{align*}
    \frac{1}{|\nabla H|} d_{t}\Phi\wedge d\gamma_{t}(v,\dot{\gamma}_{t}) &= \frac{1}{|\nabla H|} d_{t}\Phi(v) d\gamma_{t}(\dot{\gamma}_{t})\\
    &= 1,
\end{align*}
thus this is the volume form of the plane. The result follows from this.
\end{proof}

Let $\sigma$ be a closed subset of the plane. Recall that the function
\begin{equation*}
    d(\lambda,\sigma) = \inf\{|\lambda - z|\;|\;z\in\sigma\}
\end{equation*}
vanishes only if $\lambda\in\sigma$ and is Lipschitz, thus the Radamacher Theorem implies that it is differentiable almost everywhere. Furthermore, these functions are characterized by the condition
\begin{equation*}
    |\nabla H| = 1.
\end{equation*}
We now use the Coarea formula to study the integrability of the reciprocals of this kind of functions.

\begin{corollary}\label{IntegrabilidadDistancia}
Let $\sigma\subset\mathbb{C}$ be a compact set and $\Omega$ a neighborhood of $\sigma$. If $\ell(\sigma) = 0$ then
\begin{equation*}
    \lambda \longmapsto \frac{1}{d(\lambda,\sigma)}
\end{equation*}
is integrable in $\Omega\setminus\sigma$.
\end{corollary}
\begin{proof}
We first study the integrals without the assumption that $\sigma$ has a null Lebesgue measure. The function
\begin{equation*}
    d_{\sigma}(\lambda) = d(\lambda,\sigma)
\end{equation*}
satisfies $d_{\sigma}^{-1}(0) = \sigma$, thus, when used as the function $H$ in the Coarea formula, we have that
\begin{align*}
    \int_{\Omega\setminus\sigma} \frac{1}{d(\lambda,\sigma)}\;d\ell(\lambda) &= \int_{0}^{1}\int_{\gamma_{t}}\frac{1}{d(\lambda,\sigma)}\;d\gamma_{t}\;dt.
\end{align*}
Note that $\gamma_{t}$ is the set of points at a distance $t$ to $\sigma$, hence
\begin{align*}
    \int_{\Omega_{0}\setminus\sigma} \frac{1}{d(\lambda,\sigma)}\;d\ell(\lambda) &= \int_{0}^{1}\int_{\gamma_{t}}\frac{1}{t}\;d\gamma_{t}\;dt\\
    &= \int_{0}^{1}\frac{l(\gamma_{t})}{t}\;dt,
\end{align*}
where $l(\gamma_{t})$ is the length of $\gamma_{t}$. We now assume that $\ell(\sigma)=0$, which is the same as $\ell(\Omega_{0}) = 0$. Define the truncation function
\begin{equation*}
    I(\epsilon) = \int_{\epsilon}^{R}\frac{l(\gamma_{t})}{t}\;dt
\end{equation*}
with $R\geq \epsilon$ and compute the integral
\begin{align*}
    \int_{0}^{1}I(\epsilon)\;d\epsilon &= \int_{0}^{1}\int_{\epsilon}^{R}\frac{l(\gamma_{t})}{t}\;dt\;d\epsilon\\
    &\leq \int_{0}^{1}\int_{0}^{t}\frac{l(\gamma_{t})}{t}\;d\epsilon\;dt\\
    &= \int_{0}^{R}l(\gamma_{t})\;dt\\
    &= \ell(\Omega_{R}),
\end{align*}
where the last equality follows from the Coarea formula. It follows that
\begin{align*}
    I(0) &= \lim_{R\to 0}\frac{1}{R}\int_{0}^{R}I(\epsilon)\;d\epsilon\\
    &\leq \lim_{R\to 0}\frac{1}{R}\ell(\Omega_{R})\\
    &= d_{t}(\ell(\gamma_{t}))|_{t=0}\\
    &= l(\gamma_{0})\\
    &< \infty,
\end{align*}
where the last equality follows from the Coarea Formula and the previous one from the fact that $\ell(\Omega_{0}) = 0$.
\end{proof}

By means of the integral
\begin{equation*}
    \int_{0}^{1}\frac{l(\gamma_{t})}{t}\;dt
\end{equation*}
it is easy to see that the integral of $\frac{1}{d_{\sigma}}$ can diverge for sets with non-null measure, such as $\overline{B_{R}(0)}$. This does not imply that $\frac{1}{d_{\sigma}}$ diverges for sets with non-null measure, as care has to be taken not to integrate over sets where the functions are infinite, in the same way that this does not imply that $\frac{1}{|\lambda|}$ has a divergent integral near $0$ in $\mathbb{R}^{2}$.

The Cauchy-Pompeiu formula can also be used to obtain information on the integral over codimension $1$ manifolds from the original integral.

\begin{theorem}\label{TeoCPCoárea}
Let $(\gamma_{n})_{n\in\mathbb{N}}$ be a sequence of smooth curves such that $int\;\gamma_{n}\searrow \Omega$. If $f\in C^{1}(int\;\gamma_{1}\setminus\Omega)$ then the limit
\begin{equation*}
    \lim_{n\to\infty}\int_{\gamma_{n}} f\;d\gamma_{n}
\end{equation*}
exists.
\end{theorem}
\begin{proof}
By the Cauchy-Pompeiu Formula, there exists a number $c$ such that
\begin{equation*}
    c = \frac{1}{2\pi i}\int_{\gamma_{n}}f\;d\gamma_{n} - \frac{1}{\pi}\int_{int\;\gamma_{1}\setminus\Omega}\partial_{z^{\ast}}f\;d\ell
\end{equation*}
for every $n\in\mathbb{N}$. The left-hand side clearly has a limit as $n\to\infty$ and the second integral in the right-hand side vanishes on the same limit. This implies that the first integral on the right-hand side has a limit as $n\to\infty$.
\end{proof}

Note that the limit in the previous result is precisely $c$, which is the common value of the integrals
\begin{equation*}
    \frac{1}{2\pi i}\int_{\gamma}f\;d\gamma - \frac{1}{\pi}\int_{int\;\gamma\setminus\Omega}\partial_{z^{\ast}}f\;d\ell
\end{equation*}
along any of the curves.

\section{Projection Families and Integration}

\subsection{Vector Integration}

We now provide a brief summary of the most important integrals of vector-valued functions. See \cite{Diestel}, \cite{Dinculeanu} or \cite{Graves}. First of all, two notions of measurability arise naturally in this setting.

\begin{definition}
Let $(\Omega,\Sigma,\mu)$ be a measure space, $X$ a Banach space and $f\colon \Omega \to X$.
\begin{enumerate}
    \item $f$ is \textbf{weakly measurable} if the function $\Lambda\circ f$ is measurable for each $\Lambda\in X^{\ast}$.

    \item $f$ is \textbf{strongly measurable} if there exists a sequence of simple functions $(s_{n})_{n\in\mathbb{N}}$ such that $s_{n} \xrightarrow[]{a.e.} f$.
\end{enumerate}
\end{definition}

The first type of integrals comes from considering Riemann-type sums.

\begin{definition}
Let $(\Omega,\Sigma,\mu)$ be a measure space, $X$ a Banach space, and $f\colon \Omega \to X$ a strongly measurable function. $f$ is \textbf{Bochner integrable} if there exists a sequence of simple functions $(s_{n})_{n\in\mathbb{N}}$ such that
\begin{equation*}
    \lim_{n\to\infty}\int |s_{n} - f|\;d\mu = 0,
\end{equation*}
in which case we define the \textbf{Bochner integral} as
\begin{equation*}
    \int f\;d\mu = \lim_{n\to\infty} \int s_{n}\;d\mu.
\end{equation*}
\end{definition}

\begin{theorem}[Bochner's Integrability Theorem]
Let $(\Omega,\Sigma,\mu)$ be a measure space, $X$ a Banach space and $f\colon \Omega \to X$. $f$ is Bochner integrable if and only if $f$ is stringly measurable and
\begin{equation*}
    \int |f|\;d\mu < \infty.
\end{equation*}
\end{theorem}

The class of Bochner integrable functions is not large enough and its computation is not easy. The following two integrals are improvements in both senses.

\begin{definition}
Let $(\Omega,\Sigma,\mu)$ be a measure space, $X$ a Banach space, and $f\colon \Omega \to X$ a weakly measurable function. We say that $f$ is \textbf{scalarly integrable} if $\Lambda\circ f\in L^{1}(\mu)$ for each $\Lambda\in X^{\ast}$.
\end{definition}

\begin{definition}
Let $(\Omega,\Sigma,\mu)$ be a measure space, $X$ a Banach space, and $f\colon \Omega \to X$ a scalarly integrable function. The function $f$ is \textbf{Pettis integrable} if there exists $\int f\;d\mu \in X$ such that
\begin{equation*}
    \Lambda\left( P\int f\;d\mu \right) = P\int \Lambda\circ f\;d\mu
\end{equation*}
for each $\Lambda\in X^{\ast}$.
\end{definition}

\begin{theorem}[Dunford's Lemma]
Let $(\Omega,\Sigma,\mu)$ be a measure space, $X$ a Banach space, and $f\colon \Omega \to X$ a scalarly integrable function. The functional
\begin{equation*}
\begin{array}{ccc}
    D\int f\;d\mu \colon & \longrightarrow &\mathbb{C} \\
     \Lambda & \longmapsto &\int \Lambda\circ f\;d\mu
\end{array}
\end{equation*}
defines an element of $X^{\ast\ast}$.
\end{theorem}

\begin{definition}
Let $(\Omega,\Sigma,\mu)$ be a measure space, $X$ a Banach space, and $f\colon \Omega \to X$ a scalarly integrable function. We define the \textbf{Dunford integral} as the functional $D\int f\;d\mu\in X^{\ast\ast}$ given by
\begin{equation*}
    \left(D\int f\;d\mu\right)(\Lambda) = \int \Lambda\circ f\;d\mu.
\end{equation*}
\end{definition}

Schematically, the different kinds of integrability we have discussed so far are related in the following way.
\begin{table}[h]
    \centering
Bochner $\Rightarrow$ Pettis $\Rightarrow$ Dunford.
\end{table}
If $f$ is Bochner integrable then the Bochner and Pettis integrals are equal. If $f$ is Pettis integrable then the Pettis and Dunford integrals are related by a slightly more complicated equation:
\begin{equation*}
    D\int f\;d\mu = J\left(P\int f\;d\mu\right).
\end{equation*}

\subsection{Projection Families}

We now summarize the most basic properties of vector and operator projection families, as developed in \cite{ProjectionFamilies}.

\begin{definition}
Let $X$ be a Banach space and $(\Omega,\Sigma)$ a measurable space. An \textbf{operator projection family} is a family of measures in $\Omega$, denoted by
\begin{equation*}
    \mu = \{\mu_{\Lambda,x}\;|\;\Lambda\in X^{\ast},\;x\in X\},
\end{equation*}
with the following two properties.
\begin{enumerate}
    \item The function
    \begin{equation*}
    \begin{array}{ccc}
        X^{\ast} \times X & \longrightarrow &\mathcal{M}(\Omega) \\
        (\Lambda,x) & \longmapsto & \mu_{\Lambda,x}
    \end{array}
    \end{equation*}
    is bilinear.
    \item If $(\Lambda_{i})_{i\in I}$ and $(x_{j})_{j\in J}$ are nets such that $\Lambda_{i} \to \Lambda$ and $x_{j} \to x$ then
    \begin{equation*}
        \mu_{\Lambda_{i},x} \xrightarrow[]{set} \mu_{\Lambda,x}
    \end{equation*}
    and
    \begin{equation*}
        \mu_{\Lambda,x_{j}} \xrightarrow[]{set} \mu_{\Lambda,x}.
    \end{equation*}
\end{enumerate}
\end{definition}

\begin{definition}
Given an operator projection family $\mu$ and $f\in L^{1}(\mu)$ we define the \textbf{integral} of $f$ \textbf{with respect to} $\mu$ as the operator $\int f\;d\mu\in B(X,X^{\ast\ast})$ induced by the bilinear bounded form given by
\begin{equation*}
\begin{array}{cccc}
    \int f\;d\mu\colon& X^{\ast}\times X & \longrightarrow &\mathbb{C} \\
    & (\Lambda,x) & \longmapsto & \int f\;d\mu_{\Lambda,x}
\end{array}.
\end{equation*}
\end{definition}

\begin{definition}
Let $X$ be a Banach space, $(\Omega,\Sigma)$ a measurable space, and $\mu$ an operator projection family. We say that $f\in L^{\infty}(\mu)$ is \textbf{properly integrable} if
\begin{equation*}
    \int f\;d\mu \in J(X).
\end{equation*}
\end{definition}

\section{Smooth and Continuous Functional Calculi}

\subsection{Vector Cauchy-Pompeiu Formula}

We now extend the Cauchy-Pompeiu formula to the Banach-valued case.

\begin{proposition}[Vector Cauchy-Pompeiu Formula]
Let $X$ be a Banach space and $f\colon \Omega\subset\mathbb{C} \to X$ be a $C^{\infty}$ function. The Cauchy-Pompeiu Formula is valid for $f$ in the following sense
\begin{equation*}
    f(\lambda) = \frac{1}{2\pi i}P\int_{\partial\omega}\frac{f}{z-\lambda}\;d\gamma - \frac{1}{\pi}P\int_{\omega}\frac{\partial_{z^{\ast}}f}{z-\lambda}\;d\ell.
\end{equation*}
\end{proposition}
\begin{proof}
The functions inside the integrals are continuous in the punctured set $\omega_{\epsilon} = \omega\setminus B_{\epsilon}(0)$, which has compact closure, hence the Pettis integrals on the right-hand side exist over the set $\omega_{\epsilon}$. A limiting procedure similar to the one used in the proof of the Cauchy-Pompeiu formula (theorem \ref{CauchyPompeiu}) shows that the integrals in the right-hand side exist as elements of $X^{\ast\ast}$. This is because, despite the Dominated Convergence Theorem for the Pettis integral implies that the integral exists in $X$, the Monotone Convergence Theorem for Pettis integral only concludes existence as an element of $X^{\ast\ast}$. If $\Lambda\in X^{\ast}$ then $\Lambda\circ f\colon \Omega\subset\mathbb{C} \to \mathbb{C}$ is a $C^{\infty}$ function, thus the Cauchy-Pompeiu formula is valid and implies that
\begin{equation*}
    \Lambda\circ f(\lambda) = \frac{1}{2\pi i}\int_{\partial\omega}\frac{\Lambda\circ f}{z-\lambda}\;d\gamma - \frac{1}{\pi}\int_{\omega}\frac{\partial_{z^{\ast}}\Lambda\circ f}{z-\lambda}\;d\ell.
\end{equation*}
The definition of the Dunford integral and this equation imply that
\begin{equation*}
    \Lambda(f(\lambda)) = \left(\frac{1}{2\pi i}D\int_{\partial\omega}\frac{f}{z-\lambda}\;d\gamma - \frac{1}{\pi}D\int_{\omega}\frac{\partial_{z^{\ast}}f}{z-\lambda}\;d\ell\right)(\Lambda).
\end{equation*}
Since $\Lambda\in X^{\ast}$ is arbitrary, we conclude that
\begin{equation*}
    J(f(\lambda)) = \frac{1}{2\pi i}J\left(P\int_{\partial\omega}\frac{f}{z-\lambda}\;d\gamma\right) - \frac{1}{\pi}D\int_{\omega}\frac{\partial_{z^{\ast}}f}{z-\lambda}\;d\ell.
\end{equation*}
This implies that
\begin{align*}
    \frac{1}{\pi}D\int_{\omega}\frac{\partial_{z^{\ast}}f}{z-\lambda}\;d\ell &= \frac{1}{2\pi i}J\left(P\int_{\partial\omega}\frac{f}{z-\lambda}\;d\gamma\right) - J(f(\lambda))\\
    &= J\left(\frac{1}{2\pi i}P\int_{\partial\omega}\frac{f}{z-\lambda}\;d\gamma - f(\lambda)\right),
\end{align*}
hence the function $\frac{\partial_{z^{\ast}}f}{z-\lambda}$ is Pettis integrable with respect to $\ell$, thus we have
\begin{equation*}
    f(\lambda) = \frac{1}{2\pi i}P\int_{\partial\omega}\frac{f}{z-\lambda}\;d\gamma - \frac{1}{\pi}P\int_{\omega}\frac{\partial_{z^{\ast}}f}{z-\lambda}\;d\ell.
\end{equation*}
\end{proof}

The immediate corollary is that the integrals in the previous formula do not depend on the trajectory.

\begin{corollary}
Let $X$ be a Banach space and $f\colon \Omega\subset\mathbb{C} \to X$ a $C^{\infty}$ function. Given smooth curves $\gamma_{1,2}$ such that $\gamma_{1,2}\subset\Omega$ and $\lambda\in int\;\gamma_{1,2}$ then
\begin{equation*}
    \frac{1}{2\pi i}\int_{\gamma_{1}}\frac{f}{z-\lambda}\;d\gamma - \frac{1}{\pi}\int_{int\;\gamma_{1}}\frac{\partial_{z^{\ast}}f}{z-\lambda}\;d\ell = \frac{1}{2\pi i}\int_{\gamma_{2}}\frac{f}{z-\lambda}\;d\gamma - \frac{1}{\pi}\int_{int\;\gamma_{2}}\frac{\partial_{z^{\ast}}f}{z-\lambda}\;d\ell,
\end{equation*}
that is, the integrals in the Vector Cauchy-Pompeiu Formula do not depend on the trajectory.
\end{corollary}

Once again, there is another immediate consequence.

\begin{theorem}[Vector Cauchy-Pompeiu]
Let $\Omega$ be an open subset of $\mathbb{C}$ with compact closure and $f\colon \Omega \to X$ a $C^{\infty}$ function. The following statements are true:
\begin{enumerate}
    \item For any $\gamma$ simple closed curve in $\Omega$ we have
    \begin{equation*}
        \frac{1}{2\pi i}\int_{\gamma} f\;d\gamma - \frac{1}{\pi}\int_{int\;\gamma}\partial_{z^{\ast}}f\;d\ell = 0.
    \end{equation*}
    \item If $\gamma_{1}$ and $\gamma_{2}$ are closed curves homotopic in $\Omega$ then
    \begin{equation*}
        \frac{1}{2\pi i}\int_{\gamma} f\;d\gamma - \frac{1}{\pi}\int_{int\;\gamma}\partial_{z^{\ast}}f\;d\ell = \frac{1}{2\pi i}\int_{\gamma} f\;d\gamma - \frac{1}{\pi}\int_{int\;\gamma}\partial_{z^{\ast}}f\;d\ell.
    \end{equation*}
    In particular, All integrals exist in $X^{\ast\ast}$.
\end{enumerate}
\end{theorem}

\subsection{Regular Elements and Smooth Functional Calculus}

We now develop the functional calculi we will use to obtain the Spectral Theorem. We do this in the context of Banach algebras. Further ahead we apply a similar but less restrictive procedure to the case of operators in a Banach space, which, essentially, consists in applying these results in a pointwise manner, except for certain modifications. We trust the reader can provide these changes. We take this risk since the methods we develop also provide a version of the Spectral Theorem for Banach algebras.

\begin{definition}
Let $X$ be a Banach algebra and $x\in X$. We say that $x$ is \textbf{regular} if for any $\Omega$ precompact neighborhood of $\sigma(x)$ the resolvent function $R_{x}$ is Dunford integrable.
\end{definition}

We will work with the class $C^{\infty}(x)$ of $C^{\infty}$ functions defined in a precompact neighborhood of $\sigma(x)$.

\begin{definition}
Let $X$ be a Banach algebra, $x\in X$ regular, $f\in C^{\infty}(x)$ and $\gamma$ a smooth curve such that $\sigma(x) \subset int\;\gamma$ and $\gamma\subset Dom\;f$. We define $f(x)$ as the element of $X^{\ast\ast}$ given by
\begin{equation*}
    f(x) = \frac{1}{2\pi i}D\int_{\gamma}f(\lambda)R_{x}(\lambda)\;d\gamma(\lambda) - \frac{1}{\pi}D\int_{int\;\gamma\setminus \sigma(x)}\partial_{z^{\ast}}f(\lambda)R_{x}(\lambda)\;d\ell(\lambda).
\end{equation*}
The application
\begin{equation*}
\begin{array}{ccc}
    C^{\infty}(x) & \longrightarrow & X^{\ast\ast}\\
    f & \longmapsto &f(x)
\end{array}
\end{equation*}
is the \textbf{Smooth Functional Calculus}.
\end{definition}

The Vector Cauchy-Pompeiu Formula implies that the previous definition does not depend on the trajectory. Despite $f(x)\in X^{\ast\ast}$ we will write $\Lambda(f(x))$ instead of $f(x)(\Lambda)$, understanding that $f(x)$ is actually an element of the bidual space.

The Smooth Functional Calculus possesses a natural continuity property.

\begin{proposition}\label{ContinuidadCFS}
Let $X$ be a Banach algebra and $x\in X$ regular. If $\Omega$ is a neighbourhood of $\sigma(x)$ then the application
\begin{equation*}
\begin{array}{ccc}
    (C^{\infty}(\Omega),|\cdot|_{C^{1}}) &\longrightarrow &\mathbb{C}\\
    f &\longmapsto &\Lambda(f(x))
\end{array}
\end{equation*}
is continuous for each $\Lambda\in X^{\ast}$.
\end{proposition}
\begin{proof}
It is enough to establish the continuity of the second integral. If $(f_{n})_{n\in\mathbb{N}}$ is a sequence in $C^{\infty}(\Omega)$ such that $f_{n}\xrightarrow[]{C^{1}} f$ then
\begin{align*}
    \left| \int_{int\;\gamma\setminus\sigma(x)}\partial_{z^{\ast}}f_{n}\Lambda\circ R_{x}\;d\ell - \int_{int\;\gamma\setminus\sigma(x)}\partial_{z^{\ast}}f \Lambda\circ R_{x}\;d\ell\right| &\leq \int_{int\;\gamma\setminus\sigma(x)}|\partial_{z^{\ast}}f_{n} - \partial_{z^{\ast}}f||\Lambda\circ R_{x}|\;d\ell\\
    &\leq |\partial_{z^{\ast}}f_{n} - \partial_{z^{\ast}}f|_{\infty}\int_{int\;\gamma\setminus\sigma(x)}|\Lambda\circ R_{x}|\;d\ell\\
    &\leq |f_{n} - f|_{C^{1}}\int_{int\;\gamma\setminus\sigma(x)}|\Lambda\circ R_{x}|\;d\ell.
\end{align*}
Since the last integral is finite the result follows.
\end{proof}

This continuity property is natural, because of the appearance of first derivatives, but the Spectral Theorem will require a stronger property. Essentially, we require an expression for the Smooth Functional Calculus that does not depend on derivatives. Recall that the Goldstine Theorem implies that $J(X)$ is dense in $X^{\ast\ast}$, thus $f(x)$ can be realized as a limit of evaluations in $X$ in the topology $\tau_{\omega^{\ast}}$. To this end, we make the following construction.  Given $x\in X$ the spectrum $\sigma(x)$ is compact, hence, by a known result of differential topology, there exists a $C^{\infty}$ function $H$ such that
\begin{equation*}
    H^{-1}(\{0\}) = \sigma(x).
\end{equation*}
Sard's Theorem implies that the set of critical values has null Lebesgue measure, thus there exists a sequence of real numbers $(t_{n})_{n\in\mathbb{N}}$ such that $t_{n}\to 0$ and each $t_{n}$ is a regular value. Define open sets
\begin{equation*}
    \Omega_{n} = H^{-1}([0,t_{n})),
\end{equation*}
which are neighbourhoods of $\sigma(x)$, satisfy $\Omega_{n}\searrow \sigma(x)$ and their boundaries
\begin{align*}
    \gamma_{n} &= \partial\Omega_{n}\\
    &= H^{-1}(\{t_{n}\})
\end{align*}
are smooth curves. We also have $int\;\gamma_{n} = \Omega_{n}$. Denote the measure induced by the volume form of $\gamma_{n}$ by $d\gamma_{n}$. From now on, we will use this notation.

\begin{proposition}\label{CFC}
Let $x\in X$ be regular and $f\in C^{\infty}(x)$. The equation
\begin{equation*}
    \Lambda(f(x)) = \frac{1}{2\pi i}\lim_{n\to\infty}\int_{\gamma_{n}}f\;\Lambda\circ R_{x}\;d\gamma_{n}
\end{equation*}
is satisfied for each $\Lambda\in X^{\ast}$. In particular, the limit on the right-hand side exists.
\end{proposition}
\begin{proof}
By definition,
\begin{equation*}
    \Lambda(f(x)) = \frac{1}{2\pi i}\int_{\gamma_{n}}f(\lambda)\Lambda\circ R_{x}(\lambda)\;d\gamma(\lambda) - \frac{1}{\pi}\int_{\Omega_{n}\setminus \sigma(x)}\partial_{z^{\ast}}f(\lambda)\Lambda\circ R_{x}(\lambda)\;d\ell(\lambda)
\end{equation*}
and does not depend on the trajectory $\gamma_{n}$. Since $x$ is regular the second integral in the right-hand side vanishes as $n\to\infty$ by the Monotone Convergence Theorem, which implies the desired conclusion.
\end{proof}

Notice the similarity of this result with theorem \ref{TeoCPCoárea}. Before using this result, we shift our attention to another fundamental concern. In order for the Smooth Functional Calculus to be of interest we need there to be enough regular elements in a Banach algebra. The following proposition provides a sufficient condition for regularity.

\begin{proposition}
Let $x\in X$ be an element such that $\ell(\sigma(x)) = 0$ and
\begin{equation*}
    |R_{x}(\lambda)| = \frac{1}{d(\lambda,\sigma(x))}.
\end{equation*}
From the fact that  $R_{x}(\lambda)$ is Bochner integrable, it follows that $x$ is regular.
\end{proposition}
The hypothesis on the norm is satisfied, for example, by normal operators in Hilbert spaces. Thus, this result is valid, in particular, for normal operators in Hilbert spaces with null measure spectrum.
\begin{proof}
It follows from the proposition \ref{IntegrabilidadDistancia} that $\frac{1}{d(\lambda,\sigma(x))}$ is integrable, hence our hypothesis implies that $|R_{x}(\lambda)|$ is integrable, therefore $R_{x}(\lambda)$ is Bochner integrable.
\end{proof}

Another relatively common situation is that of analytic $C_{0}$-semigroup generators, since in this case there exists a non-negative function $M\geq 0$ that depends only on the argument of $\lambda$ such that
\begin{equation*}
    |R_{x}(\lambda)| \leq \frac{M(arg\;\lambda)}{|\lambda|},
\end{equation*}
and in this case $R_{x}(\lambda)$ is Bochner integrable if $M$ is integrable, even if $\sigma(x)$ has non-null Lebesgue measure.

\subsection{Continuous Functional Calculus}

The key to establishing the Spectral Theorem is to extend the Smooth Functional Calculus to continuous functions. We will do this through proposition \ref{CFC}. If $x\in X$ is regular and $f\in C(x)$ then there exists a sequence of $C^{\infty}$ functions $(f_{n})_{n\in\mathbb{N}}$, all defined in the same open subset as $f$, such that $f_{n} \xrightarrow[]{unif} f$. The natural definition for $f(x)$ is
\begin{equation*}
    \Lambda(f(x)) = \lim_{n\to\infty} \Lambda(f_{n}(x)).
\end{equation*}
Since the Cauchy-Pompeiu formula involves the derivative $\partial_{z^{\ast}}f_{n}$ this may not make sense. For this reason, the formula of proposition \ref{CFC}, which does not involve derivatives, plays a crucial role.

\begin{theorem}[Continuous Functional Calculus]\label{TeoCFC}
The Smooth Functional Calculus is continuous when its domain $C^{\infty}(\Omega)$ is equipped with the supremum norm $|\cdot|_{\infty}$ and the codomain is equipped with the topology $\tau_{\omega^{\ast}}$. In consequence, the Smooth Functional Calculus admits a linear and continuous extension to $C(\Omega)$ by density, that is, if $f\in C(\Omega)$ and $(f_{n})_{n\in\mathbb{N}}$ is a sequence in $C^{\infty}(\Omega)$ such that $f_{n}\xrightarrow[]{unif} f$ in $\Omega$ then
\begin{equation*}
    \Lambda(f(x)) = \lim_{n\to\infty} \Lambda(f_{n}(x)).
\end{equation*}
\end{theorem}
\begin{proof}
Let $f\in C^{\infty}(x)$ and $(f_{m})_{m\in\mathbb{N}}$ be a sequence in $C^{\infty}(x)$, all with a common domain $\Omega$, such that $f_{m}\xrightarrow[]{unif} f$ in $\Omega$. By proposition \ref{ContinuidadCFS} the limit
\begin{equation*}
    \lim_{n\to\infty}\int_{\gamma_{n}} \Lambda\circ R_{x}\;d\ell
\end{equation*}
exists, hence the sequence is bounded, say by a constant $M$. For $\Lambda\in X^{\ast}$ it follows that
\begin{align*}
    2\pi |\Lambda(f_{m}(x)) - \Lambda(f(x))| &\leq \lim_{n\to\infty}\sup_{x\in \gamma_{n}}|f_{m}(x) - f(x)| \int_{\gamma_{n}}\Lambda\circ R_{x}\;d\ell\\
    &\leq M \lim_{n\to\infty}\sup_{x\in \gamma_{n}}|f_{m}(x) - f(x)|\\
    &\leq M |f_{m} - f|_{\infty}.
\end{align*}
The right-hand side converges to zero, hence $\Lambda(f_{m}(x)) \to \Lambda(f(x))$. This implies that $f_{m}(x) \xrightarrow[]{\omega^{\ast}} f(x)$.

For the sake of completeness, let us verify that the extension by density can be done (the codomain is not a Banach space, thus this is not a consequence of the usual result). Let $f\in C(\Omega)$ and $(f_{n})_{n\in\mathbb{N}}$ in $C^{\infty}(\Omega)$ be such that $f_{n} \xrightarrow[]{unif} f$. In a similar way to the previous estimate, one finds that
\begin{equation*}
    2\pi |\Lambda(f_{m}(x)) - \Lambda(f_{n}(x))| \leq M |f_{m} - f_{n}|_{\infty},
\end{equation*}
where the constant $M$ depends on $\Lambda$. The right hand-side can be made arbitrarily small, hence $(\Lambda(f_{m}(x)))_{n\in\mathbb{N}}$ is Cauchy thus convergent.

Let $(g_{n})_{n\in\mathbb{N}}$ in $C^{\infty}(\Omega)$ be such that $g_{n} \xrightarrow[]{unif} g$ and let
\begin{equation*}
    r = \lim_{n\to\infty} \Lambda(f_{n}(x)) \quad\text{and}\quad s = \lim_{n\to\infty} \Lambda(g_{n}(x)).
\end{equation*}
We have the estimate
\begin{align*}
    |r - s| &\leq |r - \Lambda(f_{n}(x))| + |\Lambda(f_{n}(x)) - \Lambda(g_{n}(x))| + |\Lambda(g_{n}(x)) - s|\\
    &\leq |r - \Lambda(f_{n}(x))| + M|f_{n} - g_{n}|_{\infty} + |\Lambda(g_{n}(x)) - s|\\
    &\leq |r - \Lambda(f_{n}(x))| + M(|f_{n} - f|_{\infty} + |f - g_{n}|_{\infty}) + |\Lambda(g_{n}(x)) - s|.
\end{align*}
The right-hand side converges to zero, hence $r = s$. Thus, the limit exists and does not depend on the sequence.

The previous considerations imply that the function $f(x)\colon X^{\ast} \to \mathbb{C}$ given by
\begin{equation*}
    \Lambda(f(x)) = \lim_{n\to\infty} \Lambda(f_{n}(x))
\end{equation*}
is well defined. Since $f_{n}(x) \in X^{\ast\ast}$ and $f(x)$ is the pointwise limit of $f_{n}(x)$ we conclude by the Uniform Boundedness Principle that $f(x)\in X^{\ast\ast}$ also.

The continuity of the extension of $f\longmapsto f(x)$ is not immediate. If $f\in C(\Omega)$ and $(f_{n})_{n\in\mathbb{N}}$ in $C^{\infty}(\Omega)$ is such that $f_{n} \xrightarrow[]{unif} f$ then
\begin{align*}
    |\Lambda(f(x))| &= \lim_{n\to\infty} |\Lambda(f_{n}(x))|\\
    &\leq M\lim_{n\to\infty}|f_{n}|_{\infty}\\
    &= M|f|_{\infty},
\end{align*}
where $M$ depends on $\Lambda$. Thus, this crucial estimate remains valid when $f\in C(\Omega)$. Now assume that $(f_{n})_{n\in\mathbb{N}}$ in $C(\Omega)$ (now only continuous instead of smooth) is such that $f_{n} \xrightarrow[]{unif} f$. The previous estimate implies that
\begin{align*}
    |\Lambda(f(x)) - \Lambda(f_{n}(x))| &\leq M|f_{n} - f|_{\infty},
\end{align*}
where once again, $M$ depends on $\Lambda$. The term on the right vanishes, thus $\Lambda(f_{n}(x)) \to \Lambda(f(x))$ and $f_{n}(x) \xrightarrow[]{\omega^{\ast}} f(x)$. This establishes continuity of $f\longmapsto f(x)$ and finishes the proof.
\end{proof}

The previous constructions have the defect that the function needs to be defined in a neighborhood of $\sigma(x)$. We now verify that only the values of the function in $\sigma(x)$ are important.

\begin{lemma}
Let $f,g\in C(x)$ be such that $f\restriction_{\sigma(x)} = g\restriction_{\sigma(x)}$, then $f(x) = g(x)$.
\end{lemma}
\begin{proof}
It is enough to establish this for smooth functions. We have the estimate
\begin{align*}
    2\pi |\Lambda(f(x)) - \Lambda(g(x))| &= \left| \lim_{n\to\infty} \int_{\gamma_{n}} (f - g)\;\Lambda\circ R_{x}\;d\ell\right|\\
    &\leq \lim_{n\to\infty}\sup_{s\in\gamma_{n}}|f(x) - g(x)|\int_{\gamma_{n}}|\Lambda\circ R_{x}|\;d\ell\\
    &\leq M \lim_{n\to\infty}\sup_{s\in\gamma_{n}}|f(x) - g(x)|.
\end{align*}
Since $f$ and $g$ are continuous and equal in $\sigma(x)$ we have that the last limit vanishes by uniform continuity. Since $\Lambda\in X^{\ast}$ is arbitrary we have that $f(x) = g(x)$.
\end{proof}

If $f\in C(\sigma(x))$ the Tietze Extension Theorem implies that $f$ can be continuously extended to any neighborhood of $\sigma(x)$ with the same norm and the previous lemma implies that if $F$ and $G$ are two such extensions then $F(x) = G(x)$, hence we can define $f(x)$ as any of them. The function obtained this way is still continuous, since if $f_{n} \xrightarrow[]{unif} f$ in $\sigma(x)$ then the functions $f - f_{n}$ can be extended to a common neighborhood in such a way that the extension has the same norm. Continuity in $C(x)$ implies that $f_{n}(x) \xrightarrow[]{\omega^{\ast}} f(x)$ and continuity is established. The following proposition summarizes the previous results and gathers the different expressions that we have obtained for the Smooth and Continuous Functional Calculi.

\begin{proposition}[Smooth and Continuous Functional Calculi]
Let $X$ be a Banach algebra and $x\in X$ a regular element.
\begin{enumerate}
    \item The \textbf{Smooth Functional Calculus} is the map $C^{\infty}(x) \to X^{\ast\ast}$ given by
    \begin{align*}
        \Lambda(f(x)) &= \frac{1}{2\pi i}\int_{\gamma}f\;\Lambda\circ R_{x}\;d\ell - \frac{1}{2\pi}\int_{int\;\gamma\setminus \sigma(x)} \partial_{z^{\ast}}f\;\Lambda\circ R_{x}\;d\ell\\
        &= \lim_{n\to\infty}\frac{1}{2\pi i}\int_{\gamma_{n}}f\;\Lambda\circ R_{x}.
    \end{align*}
    This map is continuous if the codomain is equipped with the topology $\tau_{\omega^{\ast}}$.

    \item The \textbf{Continuous Functional Calculus} is the map $C(\sigma(x)) \to X^{\ast\ast}$ given by
    \begin{align*}
        f(x) &= \lim_{n\to\infty} f_{n}(x),
    \end{align*}
    where $(f_{n})_{n\in\mathbb{N}}$ is any sequence of smooth functions with a neighborhood of $\sigma(x)$ as a common domain such that $f_{n}\xrightarrow[]{unif} f$. This map is continuous when considered as a functon $(C(\sigma(x)),|\cdot|_{\infty})\to (X,\tau_{\omega^{\ast}})$.
\end{enumerate}
\end{proposition}

%%%%%%%%%%%%%%%%%%%%%%%%%%%%%%%%
\section{Spectral Theorem in Banach Spaces}

We can finally establish our generalizations of the Spectral Theorem. The first version is for regular elements of a Banach algebra. This applies, in particular, to the case of bounded linear operators in Banach spaces. Regularity is too restrictive of a condition in this context, hence our second version of the Spectral Theorem requires a less restrictive condition for operators in Banach spaces. In this case, the result will be valid for a certain class of operators called pointwise regular, which includes normal operators in Hilbert spaces.

\subsection{Spectral Theorem in Banach Algebras}

The previous result implies that the functional
\begin{equation*}
\begin{array}{ccc}
    C(\sigma(x)) & \longrightarrow &\mathbb{C}\\
     f  &\longmapsto  &\Lambda(f(x))
\end{array}
\end{equation*}
defines an element of $C(\sigma(x))^{\ast}$ for each $\Lambda\in X^{\ast}$. The Riesz-Markov-Kakutani Theorem implies that there exists a unique Borel measure $\mu^{x}_{\Lambda}$ in $\sigma(T)$ such that
\begin{equation*}
    \Lambda(f(x)) = \int_{\sigma(x)}f\;d\mu^{x}_{\Lambda}.
\end{equation*}
We have proven the following.

\begin{theorem}
Let $X$ be a Banach algebra and $x\in X$ regular. There exists a unique family of measures
\begin{equation*}
    \mu_{x} = \{\mu^{x}_{\Lambda}\;|\;\Lambda\in X^{\ast}\}
\end{equation*}
in $\sigma(x)$ such that
\begin{equation}
    \Lambda(f(x)) = \int_{\sigma(T)}f\;d\mu^{x}_{\Lambda}
\end{equation}
for each $\Lambda\in X^{\ast}$ and $f\in C(\sigma(x))$. In particular,
\begin{align*}
    \Lambda(x) &= \int_{\sigma(T)}Id\;d\mu^{x}_{\Lambda}\\
    &= \int_{\sigma(T)}\lambda\;d\mu^{x}_{\Lambda}(\lambda).
\end{align*}
\end{theorem}

We will call $\mu^{x}$ the \textbf{spectral family} of $x$. We now show that $\mu^{x}$ is a vector projection family.

\begin{proposition}
Let $X$ be a Banach algebra, $x\in X$ regular and $\mu^{x}$ its spectral family. The function
\begin{equation*}
    \Lambda \longmapsto \mu^{x}_{\Lambda}
\end{equation*}
is linear.
\end{proposition}
\begin{proof}
Given $\Lambda,\Phi\in X^{\ast}$ and $c\in\mathbb{C}$, we have the following estimate for each $f\in C(\sigma(x))$
\begin{align*}
    \int_{\sigma(x)}f\;d\mu^{x}_{\Lambda + c\Phi} &= (\Lambda + c \Phi)(f(x))\\
    &= \Lambda(f(x)) + c \Phi(f(x))\\
    &= \int_{\sigma(x)}f\;d\mu^{x}_{\Lambda} + c \int_{\sigma(x)}f\;d\mu^{x}_{\Phi}\\
    &=\int_{\sigma(x)}f\;d(\mu^{x}_{\Lambda} + c \mu^{x}_{\Phi}),
\end{align*}
that is, for each $f\in C(\sigma(x))$ we have that
\begin{equation*}
    \int_{\sigma(x)}f\;d(\mu^{x}_{\Lambda + c\Phi} - \mu^{x}_{\Lambda} - c \mu^{x}_{\Phi}) = 0.
\end{equation*}
Since this is a Borel measure, this can only happen if
\begin{equation*}
    \mu^{x}_{\Lambda + c\Phi} = \mu^{x}_{\Lambda} + c\mu^{x}_{\Phi}.
\end{equation*}
\end{proof}

It remains only to prove the continuity of the spectral family $\mu^{x}$. This will follow from the continuity of the Continuous Functional Calculus.

\begin{theorem}\label{TeoBicontNorm}
Let $X$ be a Banach algebra, $x\in X$ regular and $\mu^{x}$ its spectral family. The spectral family is a vector projection family.
\end{theorem}
\begin{proof}
If $g\in C(\sigma(x))$ then the function $\Lambda\longmapsto \Lambda(f(x))$ is linear and bounded since
\begin{equation*}
    \int_{\sigma(x)}g\;d\mu^{x}_{\Lambda} = \Lambda(g(x))
\end{equation*}
and the right-hand side is linear and bounded as a function of $\Lambda$. This implies that the same is valid for elements of $L^{\infty}(\mu^{x})$, therefore the family $\mu^{x}$ is a vector projection family.
\end{proof}

We can finally state the Spectral Theorem, the proof of which is given by all of the previous results.

\begin{theorem}[Spectral Theorem in Banach Algebras]\label{TeoEspectralAB}
Let $X$ be a Banach algebra and $x\in X$ regular. There exists a unique vector projection family $\mu^{x}$ in $\sigma(x)$ such that
\begin{equation*}
    f(x) = \int_{\sigma(x)}f\;d\mu^{x}
\end{equation*}
for each $f\in C(\sigma(x))$. In particular,
\begin{align*}
    x &= \int_{\sigma(x)}Id\;d\mu^{x}\\
    &= \int_{\sigma(x)}\lambda \;d\mu^{x}(\lambda).
\end{align*}
\end{theorem}

In virtue of the previous theorem, we can apply the theory of vector projection families to the spectral family, In particular, if $f\in L^{1}(\mu)$ the integral
\begin{equation*}
    \int_{\sigma(x)}f\;d\mu^{x}
\end{equation*}
defines an element of $X^{\ast\ast}$. This makes the following definition possible.

\begin{definition}
Let $X$ be a Banach algebra, $x\in X$ regular and $\mu^{x}$ its spectral family. We define the \textbf{Borel Functional Calculus} associated to $x$ as the function that to each $f\in L^{1}(\mu^{x})$ associates the element of $X^{\ast\ast}$ given by its integral $\int_{\sigma(T)}f\;d\mu^{x}$, that is,
\begin{equation*}
    f(x) = \int_{\sigma(x)}f\;d\mu^{x}.
\end{equation*}
\end{definition}

Note that the Borel Functional Calculus is an extension of the Holomorphic Functional Calculus. Furthermore, if $X$ is reflexive we have the following.

\begin{proposition}
Let $X$ be a Banach algebra, $x\in X$ regular and $\mu^{x}$ its spectral family. If $X$ is reflexive then each element $L^{1}(\mu^{x})$ is properly integrable.
\end{proposition}

\begin{proposition}
Let $x\in X$ be a regular element and $\mu^{x}$ its spectral family. The following conditions are equivalent:
\begin{enumerate}
    \item The family $\mu^{x}$ is induced by a vector measure $\nu^{x}$, in the sense that $\mu^{x}_{\Lambda} = \Lambda\circ \nu^{x}$ for each $\Lambda\in X^{\ast}$.

    \item Every function in $L^{\infty}(\mu^{x})$ is properly integrable.

    \item For each measurable set $E$ the function $\Lambda \longmapsto \mu^{x}_{\Lambda}(E)$ is continuous with respect to $\tau_{\omega^{\ast}}$.
\end{enumerate}

\end{proposition}

\subsection{Spectral Theorem for Operators in Banach Spaces}

The Spectral Theorem for Banach algebras can be directly applied to the Banach algebra $B(X)$. If $T\in B(X)$ then $T$ is regular if $R_{T}\colon \rho(T) \to B(X)$ is Dunford integrable, which may be too restrictive for most cases. For this reason, we modify the previous construction to improve the Spectral Theorem in this setting.

\begin{definition}
Let $X$ be a Banach space and $T\in B(X)$. We will say that $T$ is \textbf{pointwise regular} if for each $x\in X$ the function $R_{T}^{x}\colon \rho(T) \to X$ given by
\begin{equation*}
    R_{T}^{x}(\lambda) = R_{T}(\lambda)(x)
\end{equation*}
is Dunford integrable.
\end{definition}

Note that the function $R_{T}^{x}$ is always analytic in its domain and singular in $\sigma(T)$. Also note that if $T$ is regular then it is pointwise regular. When we developed the Smooth and Continuous Functional Calculi we only used the fact that $X$ was a Banach algebra to define $R_{x}$ and the fact that it is analytic in its domain. Therefore, we can define the Smooth and Continuous Functional Calculi for pointwise regular operators.

\begin{definition}[Smooth and Continuous Functional Calculi]
Let $X$ be a Banach space and $T\in B(X)$ pointwise regular.
\begin{enumerate}
    \item The \textbf{Smooth Functional Calculus} is the map $C^{\infty}(T) \to L(X,X^{\ast\ast})$ given by
    \begin{align*}
        \Lambda(f(T)(x)) &= \frac{1}{2\pi i}\int_{\gamma}f\;\Lambda\circ R_{T}^{x}\;d\gamma - \frac{1}{2\pi}\int_{int\;\gamma\setminus \sigma(x)} f\;\Lambda\circ R_{T}^{x}\;d\ell\\
        &= \lim_{n\to\infty}\frac{1}{2\pi i}\int_{\gamma_{n}}f\;\Lambda\circ R_{T}^{x}\;d\gamma_{n}.
    \end{align*}

    \item The \textbf{Continuous Functional Calculus} is the map $C(\sigma(x)) \to L(X,X^{\ast\ast})$ given by
    \begin{align*}
        \Lambda(f(T)(x)) &= \lim_{n\to\infty} \Lambda(f_{n}(x)),
    \end{align*}
    where $(f_{n})_{n\in\mathbb{N}}$ is any sequence of smooth functions with a neighborhood of $\sigma(x)$ as their common domain such that $f_{n}\xrightarrow[]{unif} f$.
\end{enumerate}
\end{definition}

The proof that the images belong to $X^{\ast\ast}$ is formally identical to the corresponding result for Banach algebras, replacing $R_{x}$ by $R_{T}^{x}$. Naturally, one must first define the Smooth Functional Calculus and then define the Continuous Functional Calculus. Once again, the proofs are identical. The only difficulty not present in the previous setting is continuity with respect to $x$.

\begin{theorem}
If $f\in C(\sigma(T))$ then $f(T)\in B(X,X^{\ast\ast})$.
\end{theorem}
\begin{proof}
Thinking of $f(T)$ as a bilinear map, it suffices to prove continuity on $x$ for fixed $\Lambda\in X^{\ast}$. The application $x \longmapsto \Lambda(f(T)(x))$ with $f\in C(\Omega)$ is the pointwise limit of $x \longmapsto \Lambda(f_{n}(T)(x))$ with $f_{n}\in C^{\infty}(\Omega)$, hence it is enough to show continuity when $f\in C^{\infty}(\Omega)$. Once again, $x \longmapsto \Lambda(f(T)(x))$ is the pointwise limit of the maps $x \longmapsto \frac{1}{2\pi}\int_{\gamma_{n}} f\Lambda\circ R_{T}^{x}\;d\gamma_{n}$ and it suffices to prove the continuity of these maps. To this end, we first show that the function $x\longmapsto \Lambda\circ R_{T}^{x}$ is continuous from $X$ to $C(\gamma_{n})$ for every $n\in\mathbb{N}$ by appealing to the Closed Graph Theorem.

Let $(x_{m})_{m\in\mathbb{N}}$ be a sequence in $X$ such that $x_{m} \to x$ and $\Lambda\circ R_{T}^{x_{n}} \xrightarrow[]{unif} h$ in $\gamma_{n}$. Given $\lambda\in\rho(T)$, we have that $R_{T}^{x_{m}}(\lambda) = R_{T}(\lambda)(x)$ and $R_{T}(\lambda)\in B(X)$, hence $R_{T}(\lambda)(x_{m}) \to R_{T}(\lambda)(x)$ and $R_{T}^{x_{m}} \xrightarrow[]{pw} R_{T}^{x}$ in $\gamma_{n}$. It follows that $\Lambda\circ R_{T}^{x_{m}} \xrightarrow[]{pw} \Lambda\circ R_{T}^{x}$ in $\gamma_{n}$. Since $\Lambda\circ R_{T}^{x_{m}} \xrightarrow[]{unif} h$ in $\gamma_{n}$ we have that $h = \Lambda\circ R_{T}^{x}$. The Closed Graph Theorem implies the continuity of this map.

With this, if $x_{m} \to x$ then $\Lambda\circ R_{T}^{x_{m}} \xrightarrow[]{unif} \Lambda\circ R_{T}^{x}$ in $\gamma_{n}$, hence $f\Lambda\circ R_{T}^{x_{m}} \xrightarrow[]{unif} f\Lambda\circ R_{T}^{x}$ in $\gamma_{n}$, thus
\begin{equation*}
    \int_{\gamma_{n}} f\Lambda\circ R_{T}^{x_{m}}\;d\gamma_{n} \to \int_{\gamma_{n}} f\Lambda\circ R_{T}^{x_{m}}\;d\gamma_{n}.
\end{equation*}
This establishes the continuity of the maps $x \longmapsto \frac{1}{2\pi}\int_{\gamma_{n}} f\Lambda\circ R_{T}^{x}\;d\gamma_{n}$ for every $n\in\mathbb{N}$ and fixed $\Lambda\in X^{\ast}$ and concludes the proof.
\end{proof}

The continuity of the Continuous Functional Calculus is now verified pointwise, that is, for each $\Lambda\in X^{\ast}$ and $x\in X$ the functional
\begin{equation}
\begin{array}{ccc}
     C(\sigma(T))& \longrightarrow &\mathbb{C}\\
     f& \longmapsto &\Lambda(f(T)(x))
\end{array}
\end{equation}
is continuous. In the case $f(T)\in B(X)$, this is just continuity in the weak operator topology. The proof is once again identical to the one for Banach algebras. Even though this property is weaker than the one for Banach algebras, it suffices to establish the Spectral Theorem.

\begin{theorem}[Spectral Theorem for Operators in Banach Spaces]\label{TeoEspectralOEB}
Let $X$ be a Banach space, and $T\in B(X)$ a pointwise regular operator. There exists a unique operator projection family $\mu^{T}$ in $\sigma(T)$ such that
\begin{equation*}
    f(T) = \int_{\sigma(T)}f\;d\mu^{T}
\end{equation*}
for each $f\in C(\sigma(T))$. In particular,
\begin{align*}
    T &= \int_{\sigma(T)}Id\;d\mu^{T}\\
    &= \int_{\sigma(T)}\lambda \;d\mu^{T}.
\end{align*}
\end{theorem}

The existence of the measures follows from the continuity of the application
\begin{equation*}
    (f,\Lambda,x) \longmapsto \Lambda(f(T)(x))
\end{equation*}
in the first variable. The fact $\mu^{T}$ is an operator projection family follows from the continuity of the last two variables. At this moment, we trust that it is clear that all of our results follow from the linearity and separate continuity of the previous function, the proofs for which we have already provided in a slightly different setting.

\begin{definition}
Let $X$ be a Banach and $T\in B(X)$ pointwise regular. We define the \textbf{Borel Functional Calculus} as the application that to each $f\in L^{1}(\mu^{T})$ associates the operator
\begin{equation*}
    f(T) = \int_{\sigma(T)} f\;d\mu^{T}
\end{equation*}
in $B(X,X^{\ast\ast})$.
\end{definition}

Once again, if $X$ is reflexive, then $f(T)$ is actually an operator in $B(X)$.

\begin{proposition}
Let $T\in B(X)$ be a pointwise regular operator and $\mu^{T}$ its spectral family. The following conditions are equivalent:
\begin{enumerate}
    \item The family $\mu^{T}$ is induced by an operator-valued measure $\nu^{T}$, in the sense that $\mu^{T}_{\Lambda,x}(E) = \Lambda(\nu^{T}(E)(x))$ for each $\Lambda\in X^{\ast}$ and $x\in X$.

    \item Each function in $L^{\infty}(\mu^{T})$ is properly integrable.

    \item For each measurable set $E$ the function $(\Lambda,x) \longmapsto \mu^{T}_{\Lambda,x}(E)$ is continuous with respect to $\tau_{\omega^{\ast}}$ in the first variable and with respect to $\tau_{\omega}$ in the second variable.
\end{enumerate}
\end{proposition}

Lastly, we show that our second version of the Spectral Theorem is an extension of the results for Hilbert spaces.

\begin{proposition}
Let $H$ be a Hilbert space and $T\in B(H)$. If $T$ is normal, then $T$ is pointwise regular.
\end{proposition}
\begin{proof}
The Spectral Theorem for operators in Hilbert spaces implies that there exists a unique resolution of the identity $E^{T}$ such that
\begin{equation*}
    \langle y,f(T)(x)\rangle = \int_{\sigma(T)}f\;dE^{T}_{y,x}
\end{equation*}
for each $f\in C(\sigma(T))$. It follows that
\begin{align*}
    \int_{\Omega\setminus\sigma(T)}|\langle y,R_{T}(\lambda)(x)\rangle|\;d\lambda &= \int_{\Omega\setminus\sigma(T)}\left|\int_{\sigma(T)}\frac{1}{z-\lambda}\;dE^{T}_{y,x}(z)\right|\;d\lambda\\
    &\leq \int_{\sigma(T)}\int_{\Omega\setminus\lambda}\frac{1}{|z-\lambda|}\;d\lambda\;d|E^{T}_{y,x}|(z)\\
    &= C |E^{T}_{y,x}|(\sigma(T)),
\end{align*}
proving the integral is finite. It follows that $R_{T}^{x}$ is Dunford integrable for each $x\in H$ and $T$; hence, pointwise regular.
\end{proof}

\section*{Acknowledgements}
This work was supported by DGAPA-UNAM, grant No. IN108225. \\

\noindent \textbf{Data Availability} 
Data sharing is not applicable to this article as no new data were created or analyzed in this study.
\\

\noindent
\textbf{Declarations}\\

\noindent
\textbf{Conflict of interest} The authors declare that they have no conflict of interest.

\nocite{*}
\bibliographystyle{alpha}
\bibliography{main}
\addcontentsline{toc}{section}{Bibliography}

\end{document}